\documentclass[a4paper,12pt]{article}     
\usepackage{amsmath,amscd,amssymb}        
\usepackage{latexsym}                     
\usepackage[english, german]{babel}                

\usepackage{theorem}                 

\newtheorem{proposition}{Proposition}

\theorembodyfont{\rmfamily}

\newenvironment{definition}
{\smallskip\noindent{\bf Definition\/}:}{\smallskip\par}

\newenvironment{remarks}
{\smallskip\noindent{\bf Remarks\/}.}{\smallskip\par}

\newenvironment{proof}{\begin{ProofwCaption}{Proof}}{\end{ProofwCaption}}
\newenvironment{proof*}[1]{\begin{ProofwCaption}{{#1}}}{\end{ProofwCaption}}
\newenvironment{ProofwCaption}[1]%
 {\addvspace\theorempreskipamount \noindent{\it #1.}\rm}%
 {\qed \par \addvspace\theorempostskipamount}
\newcommand{\qedsymbol}{\mbox{$\Box$}}
\newcommand{\qed}{\hfill\qedsymbol}

\newcommand{\Q}{{\mathbb Q}}

\newcommand{\R}{{\mathbb R}}
\newcommand{\Z}{{\mathbb Z}}

\newcommand{\RR}{{\mathcal R}}
\def\GG{{\mathcal{G}}}
\newcommand{\eps}{\varepsilon}

\newcommand{\ind}{{\rm ind}\,}

\newcommand{\chiorb}{\chi^{\rm orb}}
\newcommand{\chiun}{\chi^{\rm un}}

\title{Index of a singular point of a vector field or of a 1-form on an orbifold
}
\author{S.M.~Gusein-Zade
\thanks{
Address: Moscow State University, Faculty of mechanics and mathematics, GSP--1, Moscow, 119991, Russia.
E-mail: sabir\symbol{'100}mccme.ru}
}
\date{}

\begin{document}
\selectlanguage{english}

\maketitle

\begin{abstract}
Indices of singular points of a vector field or of a 1-form on a smooth manifold are closely related with the
Euler characteristic through the classical Poincar\'e--Hopf theorem. Generalized Euler characteristics
(additive topological invariants of spaces with some additional structures) are sometimes related with
corresponding analogues of indices of singular points.
Earlier there was defined a notion of the universal Euler characteristic of an orbifold. It takes values
in a ring $\RR$, as an abelian group freely generated by the generators, corresponding to the isomorphism
classes of finite groups. Here we define the universal index of an isolated singular point of a vector field
or of a 1-form on an orbifold as an element of the ring $\RR$. For this index, an analogue of the Poincar\'e--Hopf
theorem holds.
\end{abstract}

\section{Introduction}\label{sec:Intro}
A classical invariant of a singular point of a vector field or of a 1-form on a smooth manifold is its index.
The index is closely related with the Euler characteristic through the Poincar\'e--Hopf theorem: the sum of
indices of the (isolated) singular points of a vector field or of a 1-form on a closed (compact, without
boundary) manifold is equal to the Euler characteristic of the manifold. The notion of the index has
generalizations to vector fields and 1-forms on singular (e.~g., semianalytic) varieties which are related
in a certain form with the Euler characteristic.

The Euler characteristic (defined as the alternating sum of the ranks of the cohomology groups with compact
support) is the only, up to proportionality, additive topological invariant of smooth manifolds: see, e.~g.,
\cite{V-Euler}. For manifolds or topological spaces with additional structures one has other additive topological
invariants~--- generalized Euler characteristics. For example, for spaces with actions of a finite group $G$,
there is defined its {\em equivariant Euler characteristic} as an element of the Burnside ring $A(G)$ of the group
$G$: \cite{TtD}. The equivariant Euler characteristic is a universal additive topological invariant of manifolds
with actions of the group $G$ or of equivariant (in the sense of \cite{TtD}) $G$-CW-cmplexes: any additive
topological invariant with values in a group $\mathcal{A}$ is obtained from the equivariant Euler characteristic
via a group homomorphism $A(G)\to\mathcal{A}$. For manifolds or topological spaces with actions of the group $G$,
there are also defined {\em the orbifold Euler characteristic} $\chiorb$ (see, e.~g., \cite{AS}, \cite{HH}) and
its higher order analogues $\chi^{(k)}$ (\cite{AS}, \cite{BrF}, \cite{Tamanoi}). They can be regarded as
specifications of the equivariant Euler characteristic, i.~e.\ they are given by group homomorphisms $\chiorb$ and
$\chi^{(k)}$ from the group $A(G)$ to $\Z$. One has notions of equivariant indices of singular points of
$G$-invariant vector fields or 1-forms on $G$-manifolds or on singular $G$-spaces as elements of the Burnside ring
related with the equivariant Euler characteristic through analogues of the Poincar\'e--Hopf theorem: \cite{EG-EJM}.

A generalization of the notion of a manifold is the notion of a (real) orbifold (sometimes called a $V$-manifold):
see, e.~g., \cite{Satake1}, \cite[Appendix]{ChenRuan}. Locally an orbifold is the quotient
of a vector space by a (linear) action of a finite group (generally speaking its own for each point). In the paper
\cite{Satake2}, there were defined notions of ``the Euler characteristic of an orbifold as an orbifold'' (the
Euler--Satake characteristic) which is a rational number and of the index of a singular point of a vector field
on an orbifold (also as a rational number) so that the sum if the indices of the (isolated) singular points of a
vector field on a closed (compact, without boundary) orbifold is equal to its Euler--Satake characteristic. In
\cite{V-Euler}, there was defined {\em the universal Euler characteristic} (a universal additive topological
invariant) of orbifolds. It takes values in a ring $\RR$, as an abelian group freely generated by the elements,
corresponding to the isomorphism classes of finite groups. The ring $\RR$ is the ring of polynomials in the variables
corresponding to the isomorphism classes of indecomposable finite groups.

Here we define {\em the universal index} of an isolated singular point of a vector field or of a 1-form
on an orbifold as an element of the ring $\RR$. For this index, an analogue of the Poincar\'e--Hopf
theorem holds. One can say that the construction unites the ideas from the papers \cite{EG-EJM} and \cite{V-Euler}. 
For this index, one has an analogue of the Poincar\'e--Hopf theorem: the sum of the universal indices of the
(isolated) singular points on a closed orbifold is equal to the universal Euler characteristic of the orbifold.
The Satake index of a singular point of a vector field on an orbifold is a specification of the universal one:
it is obtained from the latter by the application of the homomorphism $\RR\to\Q$ which sents the generator
of $\RR$ corresponding to the isomorphism class of a finite group $G$ to $\frac{1}{|G|}$, wher $|G|$ is the number of
elements of $G$.

\section{The universal Euler characteristic of orbifolds}\label{sec:Orbi}
Initially a definition of an orbifold (called a $V$-manifold there) was given in \cite{Satake1}.
In more convenient terms one can find it, e.~g., in \cite[Appendix]{ChenRuan}, \cite{V-Euler}.
One can alco define the notion of an orbifold with boundary: see \cite{Satake2}, \cite[Appendix]{ChenRuan}. 

The ring $\RR$ of values of the universal Euler characteristic of orbifolds can be defined as the Grothendieck group
of finite $G$-sets for (all) finite groups $G$ with an ``induction relation''.
More precisely this means the following. For a subgroup $H$ of a finite group $G$ one has an {\em induction
operation} ${\rm Ind}_H^G$ which converts $H$-spaces into $G$-spaces. For an $H$-space $X$ the space
${\rm Ind}_H^G X$ is defined as the quotient of the Cartesian product $G\times X$ by the equivalence relation:
$(g_1,x_1)\sim(g_2, x_2)$ if there exists $h\in H$ such that $g_2=g_1 h$, $x_2=h^{-1}x_1$. The action of the group
$G$ on ${\rm Ind}_H^G X$ is defined in the natural way. The group $\RR$ is generated by the classes $[(X,G)]$
of finite $G$-sets for all finite groups $G$ modulo the relations:
\begin{enumerate}
 \item[1)] if $(X_1,G_1)\cong(X_2,G_2)$, i.~e.\ if there exist a bijective map $F:X_1\to X_2$ and a group isomorphism
 $\varphi:G_1\to G_2$ such that $F(gx)=\varphi(g)F(x)$ for $x\in X_1$, $g\in G_1$, then
 $[(X_1,G_1)]=[(X_2,G_2)]$;
 \item[2)] if $Y$ is a $G$-invariant subset of a (finite) $G$-set $X$, then
 $[(X,G)]=[(Y,G)]+[(X\setminus Y,G)]$;
 \item[3)] if $H$ is a subgroup of a finite group $G$, then, for a finite $H$-set $X$, one has
 $[(X,H)]=[({\rm Ind}_H^G X,G)]$. 
\end{enumerate}
The multiplication in $\RR$ is defined by the Cartesian product: 
$$
[(X,G)]\cdot[(X',G')]=[(X\times X',G\times G')]\,.
$$

Let us denote by $\GG$ the set of the isomorphism classes of finite groups.
It is not difficult to see that $\RR$ is the free $\Z$-module with the generators $T^{\mathfrak G}$ corresponding
to the isomorphism classes ${\mathfrak G}\in \GG$ of finite groups. The generator $T^{\mathfrak G}$ is represented
by the one-point set with the (unique) action of a representative $G$ of the class ${\mathfrak G}$. The
Krull--Schmidt theorem implies that $\RR$ is the ring of polynomials in the variables $T^{\mathfrak G}$ corresponding
to the isomorphism classes of finite indecomposable groups.

For a point $x$ of an orbifold $Q$, let us denote by ${\mathfrak G}_x\in\GG$ the class of the isotropy group of
the point $x$ (this group is defined up to isomorphism). For ${\mathfrak G}\in \GG$, let
$Q^{({\mathfrak G})}=\{x\in Q: {\mathfrak G}_x={\mathfrak G}\}$. One can see that the orbifold $Q^{({\mathfrak G})}$
is a global quotient (by a free action of a representative $G$ of the class ${\mathfrak G}$). Moreover, its
reduction is a usual ($C^{\infty}$-) manifold (with the action of the trivial group).
The representation of $Q$ as the union of the subsets $Q^{({\mathfrak G})}$ is its stratification. (We allow
stratifications with non-connected strata.) This stratification is obviously a Whitney stratification.
A homeomorphism of two orbifolds $Q_1$ and $Q_2$ is a homeomorphism of the corresponding topological spaces
preserving the described stratification, i.~e.\ mapping $Q_1^{({\mathfrak G})}$ to $Q_2^{({\mathfrak G})}$.

In order for an Euler characteristic of an orbifold (in particular, the universal one) to make sense, one needs
that the orbifold possesses certain ``finiteness properties''. For example, one can assume that it is the 
interior of a compact orbifold with boundary. For short such orbifolds are called {\em tame}. In the framework of
this paper we shall assume that all the orbifolds under consideration are tame.

\begin{definition}\label{def:universal-euler}
{\em The universal Euler characteristic} of a (tame) orbifold $Q$ is defined by the equation
\begin{equation}\label{eqn:universal-euler}
 \chiun(Q)=\sum_{{\mathfrak G}\in \GG} \chi(Q^{({\mathfrak G})}) T^{\mathfrak G} \in \RR.
\end{equation}
\end{definition}

Additive and multiplicative invariants of (tame) orbifolds are, in particular, the usual Euler characteristic $\chi$,
the Euler--Satake characteristic $\chi^{\rm ES}$, the orbifold Euler characteristic $\chi^{\rm orb}$ and its
higher order analogues $\chi^{(k)}$, $k>1$.
The universal Euler characteristic $\chiun$ is an additive and multiplicative invariant of orbifolds. Moreover,
in \cite{V-Euler} it is shown that the universal Euler characteristic is a universal additive topological invariant
of orbifolds in the sense that any additive topological invariant $I$ with values in a group $\mathcal{A}$ is
represented in the form $I=\varphi\circ\chiun$ for a unique group homomorphism $\varphi:\RR\to\mathcal{A}$.
If $\mathcal{A}$ is a ring and the invariant $I$ is multiplicative, then $\varphi$ is a ring homomorphism.
For the listed invariants the corresponding homomorphisms (which we donote by the same symbols as the invariants
themselves) are defined by the equations $\chi(T^{\mathfrak G})=1$, $\chi^{\rm ES}(T^{\mathfrak G})=\frac{1}{|G|}$,
$\chi^{\rm orb}(T^{\mathfrak G})=\chi^{\rm orb}(G/G,G)$, $\chi^{(k)}(T^{\mathfrak G})=\chi^{(k)}(G/G,G)$ for
a representative $G$ of the class ${\mathfrak G}$. Let us recall that, for a topological $G$-space $X$ its orbifold
Euler characteristic and its higher order analogues are defined by the equations (see, e.~g., \cite{AS},
\cite{Tamanoi})
$$
\chi^{\rm orb}(X,G)=\frac{1}{|G|}\sum_{{(g_1,g_2)\in G^2:}\atop{g_1g_2=g_2g_1}}
\chi\left(X^{\langle g_1,g_2\rangle}\right),
$$
$$
\chi^{(k)}(X,G)=\frac{1}{|G|}\sum_{{(g_1, \ldots g_{k+1})\in G^{k+1}:}\atop{g_ig_j=g_jg_i}}
\chi\left(X^{\langle g_1, \ldots g_{k+1}\rangle}\right),
$$
where $\langle \ldots \rangle$ is the subgroup of $G$ generated by the corresponding elements, $X^H$ is the
fixed points set of the subgroup $H\subset G$. (If the group $G$ is abelian, then $\chi^{\rm orb}(G/G,G)=|G|$,
$\chi^{(k)}(G/G,G)=|G|^k$. The usual Euler characteristic, the Euler--Satake characteristic and the orbifold
Euler characteristic can be interpreted as higher order characteristics $\chi^{(k)}$ with $k=0,\ -1$ and 1
respetively.)

For a finite group $H$, there is defined the group (!) homomorphism $r_H$ from the Burnside ring $A(H)$ to the ring
$\RR$ mapping the generator $[H/G]$ of the ring $A(H)$ ($G$ is a subgroup of $H$) to $T^{[G]}$
($[G]\in\GG$ is the isomorphism class of the group $G$). The homomorphism $r_H$ is not, in general, injective.
If the orbifold $Q$ is a global quotient: $Q=M/H$ for a smooth $H$-manifold $M$, then
$\chi^{\rm un}(Q)=r_H\left(\chi^H(M)\right)$, where $\chi^H(M)\in A(H)$ is the equivariant Euler characteristic
of the $H$-manifold $M$ in the sense of \cite{TtD} (see also \cite{SGZ-2017}).

\section{The index of a singular point on an orbifold}\label{sec:Index}
An ($n$-dimensional) orbifold $Q$ is locally (in a neighbourhood of a point $p_0$) isomorphic to the
semialgebraic set $\R^n/H$ (in a neighbourhood of its point $0$) for a linear action of a finite group $H$
on the space $\R^n$. One may assume that $\R^n/H$ is embedded into an ambient space $\R^N$. The quotient space
$\R^n/H$ (as the orbifold $Q$ itself) has a natural stratification by the subsets $(\R^n/H)^{({\mathfrak G})}$,
${\mathfrak G}\in \GG$. (The subset $(\R^n/H)^{({\mathfrak G})}$ can be nonempty only if the class ${\mathfrak G}$
contains a subgroup of $H$.)

A vector field on the orbifold $Q$ is defined locally as a (continuous) vector field on the quotient space $\R^n/H$
coordinated with the described stratification, i.~e.\ the vector at a point $x\in Q$ is tangent to the stratum
$(\R^n/H)^{(G_x)}$ containing the point $x$. (A vector tangent to a zero-dimensional stratum is, of course, equal
to zero.) A 1-form on the orbifold $Q$ is defined locally as a 1-form on the quotient space $\R^n/H$.
A germ of a vector field on $(\R^n/H, 0)$ is called {\em radial} if, in a sufficiently small punctured
neighbourhood of the point $0$, it is transversal to the spheres $S_{\varepsilon}^{N-1}$ centred at the origin
in $\R^N$ and is directed outside these spheres. An example of a radial vector field on $(\R^n/H, 0)$ is the
direct image of the vector field $\sum_{i=1}^n x_i \frac{\partial\ \ }{\partial x_i}$. A germ of a 1-form on
$(\R^n/H, 0)$ is called {\em radial}, if, for any analytic curve $\gamma:(\R_{\ge 0}, 0)\to (\R^n/H, 0)$, the 1-form
is positive on the vector $\gamma_{*}\frac{\partial\ }{\partial t}$ for sufficiently small positive
$t\in(\R_{\ge 0},0)$. An example of a radial 1-form on $(\R^n/H, 0)$ is the restriction of the form
$\sum_{i=1}^N x_i d x_i$ on $\R^N$.

For $a\in \R^n$, let $\R^n=(\R^n)^{H_a}\oplus (\R^n)^{H_a\bot}$ be a decomposition of the $H_a$-module $\R^n$
into the direct sum of its part fixed respect to the subgroup $H_a$ and of its complement. This way a neighbourhood
of the corresponding point $x$ in the quotient space $\R^n/H$ is isomorphic to the Cartesian product
$(\R^n)^{H_a}\times \left((\R^n)^{H_a\bot}/H_a\right)$.

We shall describe the corresponding constructions for vector fields. Changes for 1-forms are more or less obvious.

Assume that a vector field $U$ on $(\R^n/H,0)$ does not tend to zero in a punctured neighbourhood of the origin.
Let $B_\eps$ be the closed ball of a sufficiently small radius $\eps$ in the ambient space $\R^N$
centred at the origin such that the vector field $U$ is defined on $(\R^n/H)\cap B_{\eps}$ and does not tend
to zero at the points of $\left((\R^n/H)\cap B_{\eps}\right)\setminus \{0\}$. It is not difficult to show that there
exists a vector field $\widetilde{U}$ on $(\R^n/H)\cap B_{\eps}$ such that:
\begin{enumerate}
 \item[1)] $\widetilde{U}$ coincides with $U$ on a neighbourhood of $(\R^n/H)\cap S_{\eps}$ in
 $(\R^n/H)\cap B_{\eps}$, where $S_{\eps}=\partial B_{\eps}$ is the boundary of the ball $B_{\eps}$;
 \item[2)] in a neighbourhood of each singular point $a$ of the vector field $\widetilde{U}$ (equal to the
 Cartesian product of a neighbourhood of the point $a$ in $(\R^n)^{H_a}$ and of a neighbourhood of the origin
 in $(\R^n)^{H_a\bot}/H_a$) the latter is equal to the direct sum $U_a'\oplus U_a''$ of a vector field $U_a'$ on
 $\left((\R^n)^{H_a},0\right)$ with an isolated singular point $a$ and of a radial vector field $U_a''$ on
 $\left((\R^n)^{H_a\bot}/H_a,0\right)$.
\end{enumerate}

\begin{definition}
 {\em The universal index} of an (isolated) singular point $0$ of the vector field $U$ on $(\R^n/H,0)$ is defined as
 \begin{equation}\label{def_indun}
 \ind^{\rm un}_{\R^n/H, 0}U=
 \sum_{a\in{\rm Sing\,}\widetilde{U}} \left(\ind_{(\R^n)^{H_a},a} U_a^{'}\right)\cdot T^{{\mathfrak G}_a}\in\RR
 \end{equation}
 for the vector field $\widetilde{U}$ described above, where $\ind_{(\R^n)^{H_a},a} U_a^{'}$ is the usual index
 of the singular point of the vector field $U_a^{'}$ on the manifold $(\R^n)^{H_a}$. (If the space $(\R^n)^{H_a}$ is
 zero-dimensional (this may take place only in the case $a=0$, $H_a=H$), then the index
 $\ind_{(\R^n)^{H_a},a} U_a^{'}$ is assumed to be equal to one.)
\end{definition}

\begin{proposition}\label{prop:correct}
 The universal index $\ind^{\rm un}_{\R^n/H, 0}U$ is well defined, i.~e.\ the right hand side of Equation
 (\ref{def_indun}) does not depend on the choice of a vector field $\widetilde{U}$.
\end{proposition}

\begin{proof}
 For the proof of the statement, let us express the universal index of the vector field $U$ on $(\R^n/H,0)$ in terms
 of the radial indices of the restrictions of the vector field $U$ to certain semialgebraic subsets in $\R^n/H$.
 The fact that the radial index is well defined is proved in \cite{GD}.
 
 On the set $\GG$, let us consider the partial order $\prec$, defined by: ${\mathfrak G}_1\prec {\mathfrak G}_2$
 if (and only if) a representative of the class ${\mathfrak G}_1$ is isomorphic to a subgroup of a representative
 of the class ${\mathfrak G}_2$. Let $\zeta({\mathfrak G}_1,{\mathfrak G}_2)$be the zeta function of the
 partially ordered set $\GG$: $\zeta({\mathfrak G}_1,{\mathfrak G}_2)=1$ if
 ${\mathfrak G}_1\prec {\mathfrak G}_2$ and $\zeta({\mathfrak G}_1,{\mathfrak G}_2)=0$ in the opposite case.
 Let $\mu({\mathfrak G}_1,{\mathfrak G}_2)$ be the M\"obius function of the partially ordered set $\GG$,
 i.~e.\ the function M\"obius inverse to the zeta function:
 $$
 \sum_{{\mathfrak K}:{\mathfrak H}\prec{\mathfrak K}\prec{\mathfrak H}}
 \zeta({\mathfrak H},{\mathfrak K})\mu({\mathfrak K},{\mathfrak G})=\delta_{{\mathfrak H}{\mathfrak G}}\,,
 $$
 where $\delta_{\bullet\bullet}$ is the Kronecker delta (see, e.~g., \cite[Section~2.2]{Hall}).
 
 For ${\mathfrak G}\in \GG$, let us denote by $(\R^n/H)^{\mathfrak G}$ the union of all the strata
 $(\R^n/H)^{({\mathfrak G})}$ with ${\mathfrak G}'\prec{\mathfrak G}$. The semialgebraic set $(\R^n/H)^{\mathfrak G}$
 is closed. (It is not ,in general, an orbifold. It is the union of several (intersecting) orbifolds.) It is not
 difficult to see that the restriction of the vector field $\widetilde{U}$ to $(\R^n/H)^{\mathfrak G}$ is appropriate
 for the computation of the radial index of the restriction $U_{\vert (\R^n/H)^{\mathfrak G}}$ of the vector field
 $U$ to $(\R^n/H)^{\mathfrak G}$ according to the definition from \cite{GD}. Moreover,
 \begin{equation}\label{eqn:radial}
 {\rm ind}^{\rm rad}_{(\R^n/H)^{\mathfrak G},0}U_{\vert (\R^n/H)^{\mathfrak G}}=
 \sum_{a\in (\R^n/H)^{\mathfrak G}}
 \ind_{(\R^n)^{H_a},a} U_a^{'}=
 \sum_{{\mathfrak G}'\prec {\mathfrak G}}\sum_{a\in (\R^n/H)^{({\mathfrak G}')}}
 \ind_{(\R^n)^{H_a},a} U_a^{'}
 \,.
 \end{equation}
 
 By the M\"obius inversion theorem (\cite[Theorem~2.2.1]{Hall}) we have
 $$
 \sum_{a\in (\R^n/H)^{({\mathfrak G})}}\ind_{(\R^n)^{H_a},a} U_a^{'}=
 \sum_{{\mathfrak G}'\prec{\mathfrak G}}\mu({\mathfrak G}',{\mathfrak G})
 {\rm ind}^{\rm rad}_{(\R^n/H)^{{\mathfrak G}'},0}U_{\vert (\R^n/H)^{{\mathfrak G}'}}\,.
 $$
 Therefore
 \begin{equation}\label{eqn:indun-indrad}
 \ind^{\rm un}_{\R^n/H, 0}U=\sum_{\mathfrak G}\left(\sum_{{\mathfrak G}'\prec{\mathfrak G}}
 \mu({\mathfrak G}',{\mathfrak G})
 {\rm ind}^{\rm rad}_{(\R^n/H)^{{\mathfrak G}'},0}U_{\vert (\R^n/H)^{{\mathfrak G}'}}\right)\cdot T^{\mathfrak G}.
 \end{equation}
\end{proof}

For the universal index one has an analogue of the Poincar\'e--Hopf theorem. Let $Q$ be a closed (compact, without
boundary) orbifold and let $U$ be a vector field on $Q$ with isolated singular points (zeroes).

\begin{proposition}\label{prop:P-H}
 One has the equality
 \begin{equation}\label{eqn:uni_P-H}
 \sum_{a\in {\rm Sing\,}U}\ind^{\rm un}_{Q,a}U=\chiun(Q)\,.
 \end{equation}
\end{proposition}

\begin{proof}
The orbifold $Q$ is the disjoint union $\bigcup_{{\mathfrak G}\in\GG}Q^{({\mathfrak G})}$. Let
$Q^{{\mathfrak G}}:=\bigcup_{{\mathfrak G}'\prec{\mathfrak G}}Q^{({\mathfrak G})}$. We have
$$
\chi(Q^{{\mathfrak G}})=\sum_{{\mathfrak G}'\prec{\mathfrak G}}\chi(Q^{({\mathfrak G})})\,,
$$
which implies
$$
\chi(Q^{({\mathfrak G})})=
\sum_{{\mathfrak G}'\prec{\mathfrak G}}\mu({\mathfrak G}',{\mathfrak G})\chi(Q^{{\mathfrak G}'})\,.
$$
Therefore
\begin{equation}\label{eqn:chi_mu}
\chiun(Q)=\sum_{{\mathfrak G}\in\GG} 
\left(\sum_{{\mathfrak G}'\prec{\mathfrak G}}\mu({\mathfrak G}',{\mathfrak G})\chi(Q^{{\mathfrak G}'})\right)
\cdot T^{\mathfrak G}.
\end{equation}
According to the Poincar\'e--Hopf theorem for the (singular, locally algebraic) variety
$Q^{{\mathfrak G}'}$ (see \cite{EG-EJM}) we have
$$
\chi(Q^{{\mathfrak G}'})=
\sum_{a\in Q^{{\mathfrak G}'}} \ind^{\rm rad}_{Q^{{\mathfrak G}'},a}U_{|Q^{{\mathfrak G}'}}\,.
$$
Therefore Equations (\ref{eqn:indun-indrad}) and (\ref{eqn:chi_mu}) imply (\ref{eqn:uni_P-H}).
\end{proof}

\begin{remarks}
 {\bf 1.} The Poincar\'e--Hopf theorem for the Satake index from \cite{Satake2} is the reduction of
 Proposition~\ref{prop:P-H} by the action of the homomorphism $\chi^{\rm ES}:\RR\to\Q$.
 
 {\bf 2.} If $Q$ is a compact orbifold with the boundary $\partial Q$, then analogues of the Poincar\'e--Hopf theorem
 are the following statements: if a vector field $U$ on $Q$ with isolated singular points does not vanish on the
 boundary $\partial Q$ and is derected outside of $Q$ on it, then the sum of the universal indices of the singular
 points of $U$ is equal to the universal Euler characteristic of $Q$; if the vector field $U$ on the boundary
 is directed inside the orbifold $Q$, then the sum of the universal indices of the singular points of $U$ is equal
 to the universal Euler characteristic of the interior of $Q$. For a 1-form $\omega$ on $Q$, an analogue of the
 property of a vector field to be directed outside of (respectively inside) $Q$ is the property that, at singular points
 of the restriction of the 1-form $\omega$ to the boundary $\partial Q$, it is positive (respectively negative) on
 the vectors directed outside of $Q$.
\end{remarks}

\end{document}